\documentclass[a4paper]{article}

\usepackage[T1]{fontenc}
\usepackage{authblk}
\usepackage[utf8]{inputenc}
\usepackage{amsmath}
\usepackage{amsthm}
\usepackage{amsfonts}
\usepackage{graphicx}
\usepackage{makeidx}
\usepackage{subfigure}
\usepackage{amssymb}
\usepackage{setspace}
\usepackage[all,cmtip]{xy}
\usepackage{fancyhdr}
\usepackage{tikz}
\usepackage{tikz-cd}
\usepackage[T1]{fontenc}
\usepackage{hyperref}
\usepackage{mathtools}
\usepackage{fancyhdr}
\usepackage{pdfpages}
\usepackage{graphicx}
\usepackage{caption}
\usepackage{todonotes}

\newtheorem{Thm}{Theorem}[section]
\newtheorem{Lem}[Thm]{Lemma}
\newtheorem{Prop}[Thm]{Proposition}
\newtheorem{Cor}[Thm]{Corollary}

\newtheorem{Def}[Thm]{Definition}

\makeatletter
\def\moverlay{\mathpalette\mov@rlay}
\def\mov@rlay#1#2{\leavevmode\vtop{%
   \baselineskip\z@skip \lineskiplimit-\maxdimen
   \ialign{\hfil$\m@th#1##$\hfil\cr#2\crcr}}}
\newcommand{\charfusion}[3][\mathord]{
    #1{\ifx#1\mathop\vphantom{#2}\fi
        \mathpalette\mov@rlay{#2\cr#3}
      }
    \ifx#1\mathop\expandafter\displaylimits\fi}
\makeatother

\newcommand{\Pfin}{\mathfrak P_{\operatorname{fin}}}
\newcommand{\Pboole}{\mathfrak P_{\operatorname{Boole}}}
\newcommand{\bA}{{\mathbb A}}
\DeclareMathOperator{\Pol}{Pol}
\DeclareMathOperator{\ERP}{ERP}
\DeclareMathOperator{\HSP}{HSP}

\title{Two-element Structures Modulo Primitive Positive Constructability}

\author[1]{Manuel Bodirsky}
\affil[1]{\emph{Institute of Algebra, Technische Universit\"at Dresden, Dresden, Germany.} manuel.bodirsky@tu-dresden.de}

\author[2]{Albert Vucaj\thanks{Both authors have received funding from the European Research Council (ERC Grant
Agreement no. 681988, CSP-Infinity).}}
\affil[2]{\emph{Institute of Algebra, Technische Universit\"at Dresden, Dresden, Germany}
albert.vucaj@tu-dresden.de}

\date{}

\begin{document}
\maketitle

\begin{abstract}
Primitive positive constructions have been introduced in recent work of Barto, Opr\v{s}al, and Pinsker to study the computational complexity of constraint satisfaction problems. Let $\Pfin$ be the poset
which arises from ordering all finite relational structures by pp-constructability. This poset is infinite, but we do not know whether it is uncountable. 
In this article, we give a complete description of the restriction $\Pboole$ of $\Pfin$ to relational structures on a two-element set. We use $\Pboole$ to present the various complexity regimes of Boolean constraint satisfaction problems that were described by Allender, Bauland, Immerman, Schnoor and Vollmer. 
\end{abstract}

\section{Introduction}
Varieties play a central role in universal algebra. In 1974, Neumann~\cite{Neumann} defined the notion of interpretability between varieties, which has been studied intensively, e.g., by Garcia and Taylor~\cite{GarciaTaylor}. The corresponding lattice basically corresponds to the homomorphism order of clones.

Recently, Barto, Opr\v{s}al, and Pinsker~\cite{wonderland} introduced \emph{minor-preserving maps}, a weakening of the notion of a clone homomorphism. 
We study the poset that arises from ordering clones on a finite domain with respect to the existence of minor-preserving maps. It can be characterised in three very different, but equivalent, ways. One of the characterisations is in terms of
\emph{primitive positive constructions} for relational structures. Primitive positive constructions are also motivated by the study of the computational complexity of constraint satisfaction problems (CSPs). 
They preserve the complexity of the CSPs in the following sense: if $\bA$ and $\mathbb{B}$ are finite structures such that $\bA$ pp-constructs $\mathbb{B}$ then $\operatorname{CSP}(\mathbb{B})$  has a polynomial-time reduction to $\operatorname{CSP}(\bA)$. Barto, Opr\v{s}al, and Pinsker~\cite{wonderland} proved that $\bA$ pp-constructs $\mathbb{B}$ if and only if it exists a minor-preserving map from $\Pol({\bA})$ to $\Pol({\mathbb B})$ (Theorem~\ref{Barto}). 

Let $\Pfin$ be the poset
which arises from ordering all finite relational structures by pp-constructability. It follows from Bulatov's universal-algebraic proof~\cite{BulatovHColoring} of the $H$-coloring dichotomy theorem of Hell and Ne\v{s}et\v{r}il~\cite{HellNesetril}
that the poset $\Pfin$, restricted to all finite undirected graphs, just has three elements: ${\mathbb K}_3$ (the clique with three vertices), ${\mathbb K}_2$ (the graph consisting of a single edge),
and the graph with one vertex and a loop. 
On the other hand, the cardinality of 
$\Pfin$ is not known; it is clear that it has infinite descending chains (already for two-element structures), but it is not known whether it is uncountable.

In this article we study the restriction of $\Pfin$
to all two-element structures. 
We call this subposet $\Pboole$; it turns out that
it is a countably infinite lattice. 
We provide a description of $\Pboole$ in Theorem~\ref{mainThm}; in particular, we show that it has 3 atoms, one coatom, infinite descending chains, and a planar Hasse diagram. 
Our poset $\Pboole$ can be used to formulate a refinement of Schaefer's theorem~\cite{Schaefer} that matches the known results about the complexity of Boolean constraint satisfaction problems~\cite{AllenderSchaefer,Polymorphisms}.

\section{The PP-Constructability Poset}\label{ppposet}
As already anticipated in the introduction, 
$\Pfin$ can be defined in three different ways. In this section we introduce two of them. The third equivalent description  
relates the elements
of $\Pfin$ with classes of algebras closed not only under the classical operators H, S, and P, but also under so-called \emph{reflections}; but this will not be relevant for
the purposes of this article, so we refer the interested reader to~\cite{wonderland}.

\subsection{PP-Constructions}
Let $\tau$ be a relational signature. 
Two relational $\tau$-structures $\bA$ and $\mathbb B$ are \emph{homomorphically equivalent} if there exists a homomorphism from $\bA$ to $\mathbb B$ and vice-versa. A \emph{primitive positive formula (over $\tau$)}  is a first-order formula which only uses relation symbols in $\tau$, equality, conjunction and existential quantification. When $\bA$ is a $\tau$-structure and $\phi(x_1,\dots,x_n)$ is a $\tau$-formula with $n$ free-variables $x_1,\dots,x_n$ then 
$\{(a_1,\dots,a_n)\mid \bA \models \phi(a_1,\dots,a_n)\}$
is called the \emph{the relation defined by $\phi$}. 
If $\phi$ is primitive positive, then this relation is said to be \emph{pp-definable} in $\bA$. Given two relational structures $\bA$ and $\mathbb B$ on the same domain $A = B$ (but with possibly different signatures), we say that $\bA$ \emph{pp-defines} $\mathbb B$ if every relation in $\mathbb B$ is pp-definable in $\bA$. We say that $\mathbb B$ is a \emph{pp-power} of $\bA$ if it is isomorphic to a structure with domain $A^n$, where $n\geq 1$, whose relations are pp-definable from $\bA$ (a $k$-ary relation on $A^n$ is regarded as a $kn$-ary relation on $A$).

\begin{Def}[\cite{wonderland}]
Let $\bA$ and $\mathbb B$
be relational structures. We say that $\bA$ \emph{pp-constructs} $\mathbb B$, in symbols $\bA \preceq \mathbb B$, if $\mathbb B$ is homomorphically equivalent to a pp-power of $\bA$.
\end{Def}

The following result from~\cite{wonderland} asserts that pp-constructability preserves the complexity of $\operatorname{CSP}$s:

\begin{Prop}
Let $\bA$ and $\mathbb B$ be relational structures. If $\bA$ pp-constructs $\mathbb
B$ then $\operatorname{CSP}(\mathbb B)$ is log-space reducible to $\operatorname{CSP}(\bA)$.
\end{Prop}

Since pp-constructability is a reflexive and transitive relation on the class of relational structures~\cite{wonderland}, the relation $\equiv$ defined
by
\begin{equation*}
\bA \equiv \mathbb B 
\; :\Leftrightarrow \; \mathbb B \preceq \bA \wedge \bA \preceq \mathbb B  
\end{equation*}
is an equivalence relation. The equivalence classes of $\equiv$ are called the \emph{pp-constructability types} and we write $\overline{\bA}$ 
for the pp-constructability type of $\bA$. For any two relational structures $\bA$ and $\mathbb B$ we write $\overline{\bA}\preceq \overline{\mathbb B}$ if and only if $\bA\preceq \mathbb B$. The poset 
\begin{equation*}
\Pfin\coloneqq (\{\overline{\bA}\mid \bA\text{ is a finite relational structure}\};\preceq)
\end{equation*}
is called the \emph{pp-constructability poset}.

This article is dedicated to the subposet $\Pboole$ of $\Pfin$, whose universe is the set of all pp-constructability types of relational structures on $\{0,1\}$.

\subsection{Minor-Preserving Maps}\label{h1section}
Another approach to the pp-constructability poset involves a weakening of the notion of \emph{clone homomorphism} and certain identities called \emph{height 1 identities}. 

\begin{Def}
Let $\tau$ be a functional signature. An identity is said to be a \emph{height 1 identity} if it is of the form
\begin{equation*}
f(x_{\pi(1)},\dots,x_{\pi(n)}) \approx g(x_{\sigma(1)},\dots,x_{\sigma(m)})
\end{equation*}
where $f,g$ are function symbols in $\tau$ and $\pi$ and $\sigma$ are mappings,  $\pi\colon\{1,\dots,n\} \rightarrow \{1,\dots,r\}$ and $\sigma\colon\{1,\dots,m\} \rightarrow \{1,\dots,r\}$.
\end{Def}
 
In other words, we require that there is exactly one occurrence of a function symbol on both sides of the equality. The use of nested terms is forbidden. Identities of the form $f(x_1,\dots,x_n)\approx y$ are forbidden as well (identities of this form are often called \emph{linear} or \emph{of height at most 1}). 

A \emph{height 1 condition} is a finite set $\Sigma$ of height 1 identities. We say that a set of functions $F$ (for instance a clone) \emph{satisfies} a height 1 condition $\Sigma$, and we write $F\models\Sigma$, if for each function symbol $f$ appearing in $\Sigma$, there exists a function $f^F\in F$ of the corresponding arity such that every identity in $\Sigma$ becomes a true statement when the symbols of $\Sigma$ are instantiated by their counterparts in $F$.

We introduce some height 1 conditions that we are going to use later. 

\begin{Def} Let $f$ be a $k$-ary operation symbol.
\begin{itemize}
    \item The set of height 1 identities
\begin{equation*}
f(x,\dots,x,y) \approx f(x,\dots,y,x) \approx \dots \approx f(y,x,\dots,x) \approx f(x,\dots,x)
\end{equation*}
is called \emph{quasi near-unanimity condition} $(\operatorname{QNU}(k))$.
\item The $\operatorname{QNU}(3)$ condition is also called \emph{quasi majority condition}.
\item The set of height 1 identities ($k=3$)
\begin{equation*}
f(x,y,y) \approx f(y,x,y) \approx f(y,y,x) \approx f(x,x,x)
\end{equation*}
is called \emph{quasi minority condition}.
\end{itemize}
\end{Def}
A $k$-ary function $f$ is a \emph{quasi near-unanimity operation} if $\{f\}$ satisfies the quasi near-unanimity condition $\operatorname{QNU}(k)$. A \emph{quasi majority} and a \emph{quasi minority} operation is defined in the same way.

Let $f$ be any $n$-ary operation. Let $\pi$ be a function from $\{1,\dots,n\}$ to $\{1,\dots,r\}$. We denote by $f_\pi$ the following $r$-ary operation 
\begin{equation*}
f_\pi(x_1,\dots,x_r) \coloneqq f(x_{\pi(1)},\dots,x_{\pi(n)}).
\end{equation*}

\begin{Def}[\cite{wonderland}]
Let $\mathcal A$ and $\mathcal B$ be clones and let $\alpha \colon \mathcal A \rightarrow \mathcal B$ be a mapping that preserves arities. We say that $\xi$ is a \emph{minor-preserving map} if
\begin{equation*}
\xi(f_\pi) = \xi(f)_\pi
\end{equation*}
for any $n$-ary operation $f \in \mathcal A$ and $\pi\colon \{1,\dots,n\} \rightarrow \{1,\dots,r\}$.
\end{Def}

We write $\mathcal A\preceq_m \mathcal B$ if there exists a minor-preserving map $\xi\colon \mathcal A \rightarrow \mathcal B$, and we denote by $\equiv_m$ the equivalence relation where
$\mathcal A \equiv_m \mathcal B$ if $\mathcal A \preceq_m \mathcal B$ and $\mathcal B \preceq_m\mathcal A$. Again, we denote by $\overline{\mathcal A}$ the $\equiv_m$-class of $\mathcal A$ and we write $\overline{\mathcal A}\preceq_m\overline{\mathcal B}$ if and only if $\mathcal A\preceq_m\mathcal B$.
The connection between pp-constructability and
minor-preserving maps is given by the following theorem. 

\begin{Thm}[\cite{wonderland}]\label{Barto}
Let $\bA$ and $\mathbb B$ be finite relational structures and let $\mathcal A$ and $\mathcal B$ be their polymorphism clones. Then the following are equivalent:
\begin{enumerate}
	\item $\bA\preceq\mathbb B$.
	\item $\mathcal A\preceq_m\mathcal B$.
	\item Every height 1 condition that holds in $\mathcal A$ also holds in $\mathcal B$. 
	\item $\mathcal B \in \operatorname{ERP} \mathcal A$. 
\end{enumerate}
\end{Thm}
We refer to~\cite{wonderland} for the definitions involved in Item 4 of the statement; we just mention that $\ERP \mathcal A$ contains the universal-algebraic variety $\HSP \mathcal A$.

Note that Theorem~\ref{Barto} provides an important tool to prove that two elements are distinct in our poset:  if $\bA \npreceq \mathbb B$, then there is a height 1 condition 
which is satisfied in $\mathcal A$ but not in $\mathcal B$. Also note that every operation 
clone on a finite set is the polymorphism clone of a 
finite relational structure. Therefore, the poset
\begin{equation*}
(\{\overline{\mathcal{C}} \mid \mathcal C\text{ is an operation clone on a finite set}\};\preceq_m)
\end{equation*}
is isomorphic to $\Pfin$. In fact, both posets will be called $\Pfin$ and we will use the same symbol $\preceq$ both for the pp-constructability order between structures and for the minor-preserving order between clones.

\subsection{Post's Lattice}\label{postlattice}
The set of clones on the Boolean set $\{0,1\}$ was first investigated by Post~\cite{Post} in $1941$. This set has countably many elements and forms a lattice with respect to the inclusion order. Since we built on this result, we dedicate a section to Post's lattice in order to fix some notation. 
Note that if $\mathcal C \subseteq \mathcal D$, then it follows that $\mathcal C \preceq \mathcal D$ via the identity mapping.

\begin{figure}
	\vspace{-2cm}
	\centering
		\includegraphics[scale=0.7]{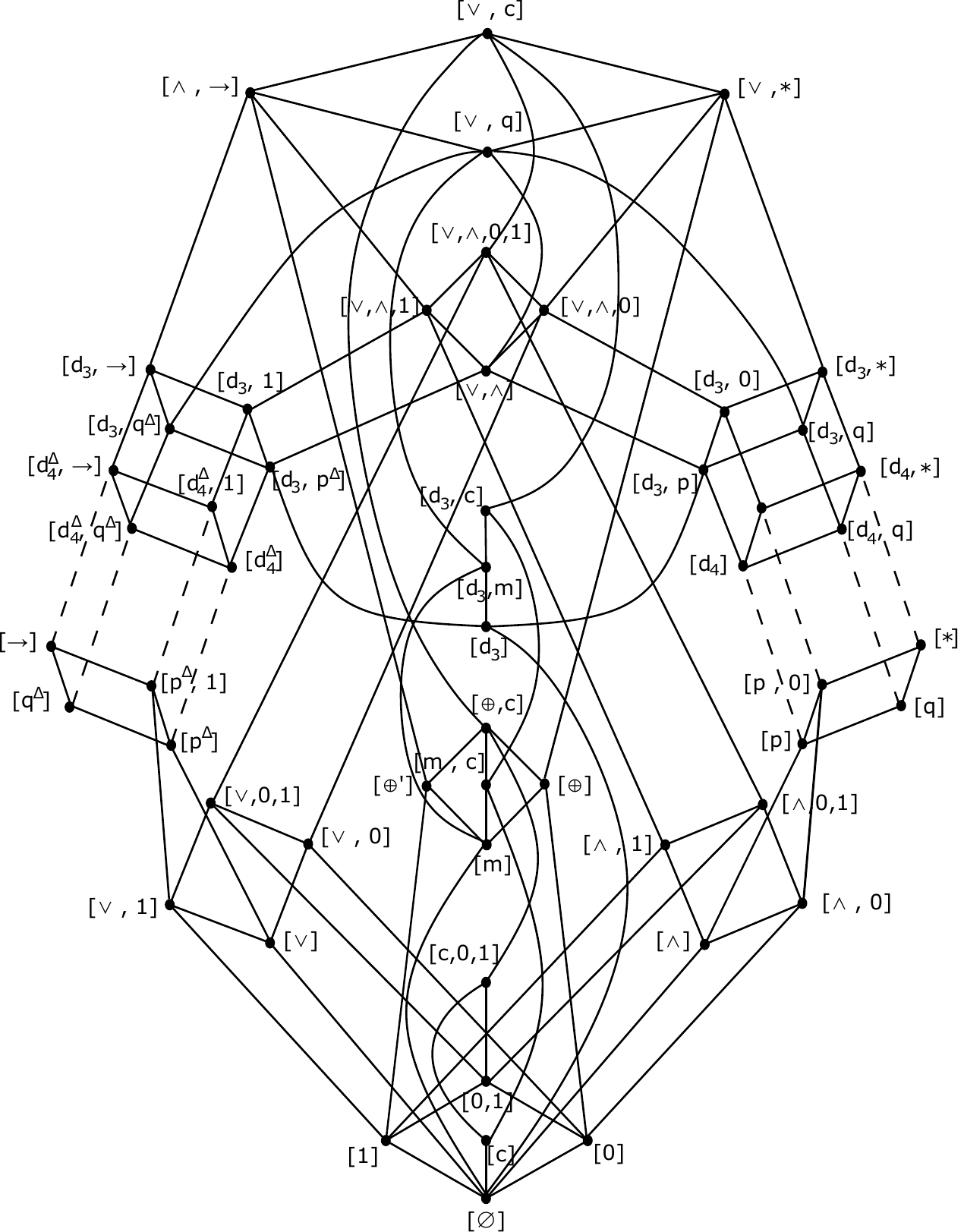}
		\caption{Post's Lattice}
		\label{fig:Post}
\end{figure}
	
We label the clones of Post's lattice by generators: if $f_1,\dots,f_n$ are operations on $\{0,1\}$, then $[f_1,\dots,f_n]$ denotes the clone generated by $f_1,\dots,f_n$. 
As usual, we may apply functions componentwise,
i.e., if $f$ is a $k$-ary map, and $\mathbf{t}_1,\dots,\mathbf{t}_k \in \{0,1\}^m$, then $f(\mathbf{t}_1,\dots,\mathbf{t}_k)$ denotes the $m$-tuple
\begin{equation*}
 (f(t_{1,1},\dots,t_{1,k}),\dots,f(t_{m,1},\dots,t_{m,k})).  
\end{equation*}
 In the description of Post's lattice, we use the following operations. 
\begin{itemize}
	\item $0$ and $1$ denote the two unary constant operations.
	\item $c(x)$ denotes the usual Boolean complementation, i.e., the non-identity permutation on $\{0,1\}$. 
	\item If $f(x_1,\dots,x_n)$ is an $n$-ary operation, then $f^\Delta(x_1,\dots,x_n)$ denotes its \emph{dual}, given by $f^\Delta(x_1,\dots,x_n) := c(f(c(x_1),\dots,c(x_n))))$.
	\item $x \oplus y := (x + y)\mod{2}$ and $x \oplus' y := c(x \oplus y)$.
	\item $x \rightarrow y := c(x) \vee y$ and $x \ast y := c(x) \wedge y$.
	\item $d_n(x_1,\dots,x_n) := \bigvee_{i = 1}^n \bigwedge_{j=1 , j\neq i} x_j$. For $n = 3$ we obtain the \emph{majority operation} 
	$d_3(x,y,z) = (x \wedge y) \vee (x \wedge z) \vee (y \wedge z)$.
	\item The \emph{minority operation} $m(x,y,z) := x \oplus y \oplus z$.
	\item $p(x,y,z) := x \wedge (y \vee z)$.
	\item $q(x,y,z) := x \wedge (y \oplus' z) = x \wedge ((y \wedge z) \vee (c(y) \wedge c(z)))$.
\end{itemize}

Post's lattice has 7 atoms, 5 coatoms and it is countably infinite because of the presence of some infinite descending chains; see Figure~\ref{fig:Post}. 

\section{The Lattice $\Pboole$}\label{ourlattice}
We consider the order defined in Section~\ref{h1section} restricted to the class of Boolean clones. Note that this is a coarser order than the usual inclusion order on the class of Boolean clones. 
In this section we are going to describe systematically the poset
\begin{equation*}
\Pboole\coloneqq (\{\overline{\mathcal C}\mid \mathcal C \text{ is a clone on } \{0,1\}\};\preceq).
\end{equation*}
First, we prove that certain Boolean clones are in the same $\equiv$-class. Later, in Section~\ref{Sep} we prove that certain Boolean clones lie in different $\equiv$-classes and, for each separation, we provide a height 1 condition as a witness.
We write 
$\mathcal C \mid \mathcal D$ if $\mathcal C\npreceq\mathcal D$ and $\mathcal D\npreceq\mathcal C$. 
Sometimes it will be useful to refer to relational descriptions of the clones: recall from Theorem~\ref{Barto} that if $\mathcal A = \Pol(\bA)$ and $\mathcal B = \Pol(\mathbb B)$, then $\mathcal A  \preceq \mathcal B$ if and only if $\bA \preceq \mathbb B$.
Finally, in Section~\ref{picture}, we display an order diagram of $\Pboole$.

\subsection{Collapses}\label{collapses}
In this section we prove that certain
clones on $\{0,1\}$ are 
in the same $\equiv$-class, i.e., represent the same element in $\Pboole$. 
We start with the observation that 
 each clone collapses with its dual. 
 
\begin{Prop}
Let $\mathcal C$ and $\mathcal D$ be two Boolean clones such that $\mathcal C\coloneqq [f_1,\dots,f_m]$ and $\mathcal D \coloneqq [f_1^\Delta,\dots,f_m^\Delta]$. Then $\mathcal C\equiv \mathcal D$.
\begin{proof}
To prove that $\mathcal C \preceq\mathcal D$, define $\xi(f):=f^\Delta$
for any $f\in \mathcal C$. Then 
\begin{align*}
  \xi(f_\pi)(x_1,\dots,x_n)& = f_\pi^\Delta(x_1,\dots,x_n) = c(f_\pi(c(x_1,\dots,x_n)) &\\ &= c(f(c(x_{\pi(1),\dots,x_{\pi(n)}}))) = c(f(c(x_{\pi(1)},\dots,x_{\pi(n)}))) &\\ & = f^\Delta(x_{\pi(1)},\dots,x_{\pi(n)}) = \xi(f)(x_{\pi(1)},\dots,x_{\pi(n)}) &\\ & =  \xi(f)_\pi(x_1,\dots,x_n).
\end{align*}
The same argument can be used to prove that $\mathcal D \preceq\mathcal C$.
\end{proof}
\end{Prop}

\begin{Prop}\label{top}
Let $\mathcal C$ be any clone and $\mathcal D$ be a clone with a constant operation. Then $\mathcal C \preceq \mathcal D$.
\begin{proof}
Note that $\mathcal D$ contains a constant operation $g^n$ for every arity $n$. The map $\xi\colon \mathcal C \rightarrow \mathcal D$ that sends every $n$-ary operation to $g^n$ is minor-preserving. 
\end{proof}
\end{Prop}

It follows that the top-element in $\Pboole$ is the class of clones that contain a constant operation. The next proposition is about the bottom element. 

\begin{Prop}\label{bottom}
Let $[\emptyset]$ be the set of projections and let $[c]$ be the clone generated by the Boolean negation. Then $[\emptyset] \equiv [c]$.
\begin{proof}
We only have to prove that $[c]\preceq [\emptyset]$. Then the map that fixes the projections and maps $c(\pi_i^{(n)})$ to $\pi_i^{(n)}$ is a minor-preserving from
$[c]$ to $[\emptyset]$.
\end{proof}
\end{Prop}

Note that with the collapses we have reported so far we can make some observations on the number of atoms in $\Pboole$. We already pointed out that $\overline{[0]}$ and $\overline{[1]}$ are not atoms in $\Pboole$, since $\overline{[0]}=\overline{[1]}$ is the top-element in $\Pboole$. Furthermore, we have that $[\vee] \equiv [\wedge]$ because they are dual to each other. Altogether, we get that $\Pboole$ has at most
three atoms: $\overline{[\wedge]}$, 
$\overline{[m]}$, 
and $\overline{[d_3]}$. We prove in Subsection~\ref{Sep} that these are distinct elements in $\Pboole$. 

Another case of collapse is that the clones $[\vee,\wedge]$ and $[d_3,p]$ represent the same element in $\Pboole$. 
We consider the binary relations
\begin{equation*}
\leq\ \coloneqq \{(0,0),(0,1),(1,1)\} 
\text{ and } B_2 := \{(0,1),(1,0),(1,1)\}
\end{equation*}
and define  (following the notation in~\cite{Polymorphisms}):
\begin{align*}
\mathbb{B}_2 & \coloneqq (\{0,1\};\{0\}, \{1\}, B_2) \\
\mathbb{D}_{\mathrm{STCON}} & \coloneqq (\{0,1\};\{0\}, \{1\}, \leq)
\end{align*}
where $\{0\}$ and $\{1\}$ are unary relations.
It is known that $[\vee,\wedge] = \Pol(\mathbb{D}_{\mathrm{STCON}})$ and $[d_3,p] = \Pol(\mathbb{B}_2, \leq)$ (see, e.g.,~\cite{Schaefer}).

\begin{Prop}
$[\vee,\wedge] \equiv [d_3,p]$.
\begin{proof}
Since $[d_3,p] \subseteq [\vee,\wedge]$ it follows that $[d_3,p]\preceq [\vee,\wedge]$. For the other inequality it suffices to prove that $(\mathbb{B}_2, \leq)$ is homomorphically equivalent to a pp-power of $\mathbb{D}_{\mathrm{STCON}}$. We consider the relational structure $\bA$ with domain $\{0,1\}^2$ and relations defined by
\begin{align*}
\Phi_0(x_1,x_2) &\coloneqq (x_1 = 0) \wedge (x_2 = 1)  \\
\Phi_1(x_1,x_2) &\coloneqq (x_1 = 1) \wedge (x_2 = 0)  \\
\Phi_\leq(x_1,x_2,y_1,y_2) &\coloneqq (x_1 \leq y_1) \wedge (y_2 \leq x_2) \\
\Phi_{B_2}(x_1,x_2,y_1,y_2) &\coloneqq  x_2 \leq y_1 .
\end{align*}
Note that $\bA$ is indeed a pp-power of $\mathbb{D}_{\mathrm{STCON}}$. We define the map 
\begin{equation*}
    f\colon \bA \to (\mathbb{B}_2, \leq)
\end{equation*}
as follows:
\begin{align*}
    f((0,1))\coloneqq 0;~f((0,0))\coloneqq f((1,0))\coloneqq f((1,1))\coloneqq 1.
\end{align*}
Let $g\colon (\mathbb{B}_2, \leq) \to \bA$ be a map such that $g(0)\coloneqq (0,1)$ and $g(1)\coloneqq (1,0)$. It is easy to check that both $f$ and $g$ are homomorphisms. This proves that $\bA$ is homomorphically equivalent to $(\mathbb{B}_2, \leq)$.
\end{proof}
\end{Prop}

\begin{figure}
    \centering
    \begin{tikzpicture}
    \node at (0,0) (0) {0};
    \node at (0,2) (1) {1};
    
    \draw[->] (0) -- (1);
    \draw [->] (0) edge[loop below] (0);
    \draw [->] (1) edge[loop above] (1);
    
\node at (1.5,1) {$\preceq$};

    \node at (3,0) (00) {00};
    \node at (5,0) (01) {01};
    \node at (3,2) (10) {10};
    \node at (5,2) (11) {11};
    \draw [->] (00) edge[loop below] (00);
    \draw [->] (10) edge[loop above] (10);
    \draw [->] (01) edge[loop below] (01);
    \draw [->] (11) edge[loop above] (11);
    \draw[dashed] [->] (10) edge[loop left] (10);
    \draw[dashed] [->] (00) edge[loop left] (00);
    \draw[dashed] [->] (11) edge[loop right] (11);
    \draw [->] (00)--(10);
    \draw [->] (01)--(00);
    \draw [->] (01) edge[bend left](10);
    \draw [->] (01)--(11);
    \draw [->] (11)--(10);
    \draw[dashed] [->] (00)--(11);
    \draw[dashed] [<->] (10) edge[bend left](01);
     \draw[dashed] [<->] (10) edge[bend left](11);
     \draw[dashed] [->] (00) edge[bend right](01);
     \draw[dashed] [<->] (00) edge[bend left](10);
     \draw[dashed] [->] (01) edge[bend right](11);

\draw (0) circle (6.5pt);
\draw (1) circle (6.5pt);
\draw (00) circle (6.5pt);
\draw (01) circle (6.5pt);
\draw (10) circle (6.5pt);
\draw (11) circle (6.5pt);

\node at (6.5,1) {$\equiv$};

\node at (8,0) (a) {0};
\node at (8,2) (b) {1};

\draw (a) circle (6.5pt);
\draw (b) circle (6.5pt);

\draw [->] (a) edge[loop below] (a);
\draw[dashed] [->] (b) edge[loop right] (b);
\draw [->] (b) edge[loop above] (b);
\draw [->] (a) edge[bend left](b);
\draw[dashed] [<->] (a) edge[bend right] (b);

\node at (0,-1) {$\mathbb{D}_{\mathrm{STCON}}$};
\node at (4,-1) {$\bA$};
\node at (8,-1) {$(\mathbb{B}_2, \leq)$};
    \end{tikzpicture}
    \caption{The structure $\mathbb{D}_{\mathrm{STCON}}$ pp-constructs $(\mathbb{B}_2, \leq)$. 
The dashed arrows connect elements in $B_2$; the solid arrows connect elements in $\leq$.}
    \label{fig:stcon}
\end{figure}
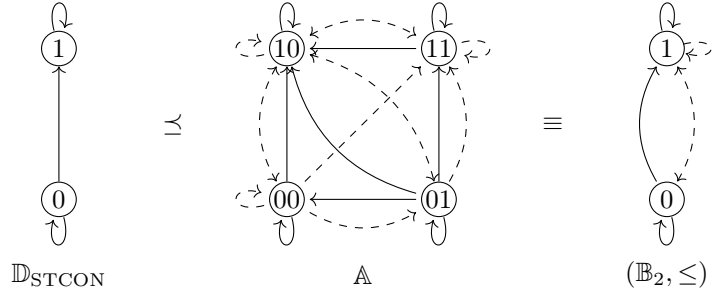

Recall that the \emph{idempotent reduct} of a clone $\mathcal C$ is the clone $\mathcal C^{\mathrm{id}}$ that consists of all idempotent operations in $\mathcal C$.

\begin{Lem}\label{lemmaId}
Let $\mathcal C$ be a Boolean clone with no constant operations. Let $\mathcal D \coloneqq [\mathcal C \cup \{c\}]$ be the clone generated by $\mathcal C$ and the Boolean negation $c$. Then we have $\mathcal D\equiv \mathcal D^{\mathrm{id}}$.
\end{Lem}
\begin{proof}
Since $\mathcal C$ contains no constant operations, for any operation $f$ in $\mathcal C$ either $f(x,\dots,x) \approx x$ holds or $f(x,\dots,x) \approx c(x)$ holds. We claim that there exists a minor-preserving map $\xi: \mathcal D \rightarrow \mathcal D^{\mathrm{id}}$. We define $\xi \colon \mathcal D \to \mathcal D^{\mathrm{id}}$ as follows: for an $n$-ary operation $f \in \mathcal D$ 
\begin{equation*}
\xi(f)(x_1,\dots,x_n)\coloneqq
   \begin{cases}
	 \ f(x_1,\dots,x_n) & \mbox{if } f \mbox{ is idempotent}
	\\ \ c(f(x_1,\dots,x_n)) & \mbox{otherwise.}
	 \end{cases}
\end{equation*}
By definition $\xi(f) \in \mathcal D^{\mathrm{id}}$. We claim that $\xi$ is minor preserving: 
if $f$ is idempotent, then $\xi$ is the identity, and the claim trivially holds; in the other case, the claim follows by the definition of negation:
\begin{align*}
\xi(f_\pi)(x_1,\dots,x_n)&\approx (c\ \circ f_\pi)(x_1,\dots,x_n) \approx (c\ \circ f)(x_{\pi(1)},\dots,x_{\pi(n)})\\ & \approx \xi(f)(x_{\pi(1)},\dots,x_{\pi(n)}) \approx \xi(f)_\pi(x_1,\dots,x_n).\qedhere
\end{align*}
\end{proof}

\begin{Cor}
$[d_3,c] \equiv[d_3,m]$ and $[m,c] \equiv[m]$.
\end{Cor}

\begin{proof}
Checking in Post's lattice $\mathfrak P$ we have that that $[d_3,c]^{\mathrm{id}} = [d_3,m]$ and $[m,c]^{\mathrm{id}} = [m]$. The statement follows from Lemma~\ref{lemmaId}.
\end{proof}

\subsection{Separations}\label{Sep}
Recall that if $\mathcal A\npreceq\mathcal B$ then there is a height 1 condition $\Sigma$ which is satisfied by some operations in $\mathcal A$ but by none of the operations in $\mathcal B$. 
In this case we say that $\Sigma$ \emph{is a witness for} $\mathcal A\npreceq\mathcal B$.
We will use the following height 1 conditions. 

\begin{Def}
The following set of height 1 identities
\begin{align*}
t_0(x,y,z)& \approx t_0(x,x,x) \\ 
t_n(x,y,z)& \approx t_n(z,z,z) \\
t_i(x,y,x)& \approx t_i(x,x,x)& 0\leq i\leq n \\ 
t_i(x,x,z)& \approx t_{i+1}(x,x,z)& \mbox{for even } i  \\ 
t_i(x,z,z)& \approx t_{i+1}(x,z,z)& \mbox{for odd } i
\end{align*}
is called \emph{$n$-ary quasi J\'onsson condition}, $\operatorname{QJ}(n)$.
\end{Def}

\begin{Prop}\label{Jonsson}\label{prop:jon}
The quasi J\'onsson condition $\operatorname{QJ}(4)$ is a witness for $[p] \npreceq [\wedge]$.
\end{Prop}

\begin{proof}
Define the operations:
\begin{align*}
    t^{[p]}_0(x,y,z) & \coloneqq p(x,x,x)
    & t^{[p]}_1(x,y,z) & \coloneqq p(x,y,z) \\
    t^{[p]}_2(x,y,z) & \coloneqq p(x,z,z)
    & t^{[p]}_3(x,y,z) & \coloneqq p(z,x,y) \\
    t^{[p]}_4(x,y,z) & \coloneqq p(z,z,z); 
\end{align*}
then $t^{[p]}_0,\dots,t^{[p]}_4$ are witnesses for $[p]\models\operatorname{QJ}(4)$.
On the other hand, let us suppose that there are witnesses $t^{[\wedge]}_0,\dots,t^{[\wedge]}_4$ for $[\wedge]\models\operatorname{QJ}(4)$ . Since any operation in $[\wedge]$ is idempotent, we have $t^{[\wedge]}_0(x,y,z) = x$. Moreover, from the identity 
\begin{align*}
    t_i(x,y,x) \approx t_i(x,x,x) && 0\leq i\leq n
\end{align*}
 we have that $t^{[\wedge]}_1(x,y,z)$ does not depend on the second argument. Moreover, $t_0(x,x,z)\approx t_1(x,x,z)$ implies that $t^{[\wedge]}_1(x,y,z)$ does not depend on the third argument. Hence, we conclude that $t^{[\wedge]}_1(x,y,z) = x$. From the identity $t_1(x,z,z)\approx t_2(x,z,z)$ and using $t^{[\wedge]}_1(x,y,z) = x$, we get that $t^{[\wedge]}_2(x,y,z) = x$. Similarly, we obtain also that $t^{[\wedge]}_3(x,y,z) = x$ and $t^{[\wedge]}_4(x,y,z) = x$. This is in contradiction with the identity $t_4(x,y,z)\approx t_4(z,z,z)$. Hence, we conclude that $[\wedge]$ does not satisfy $\operatorname{QJ}(4)$.
\end{proof}

The following structures are useful in the next proposition: 
\begin{align*}
    \mathbb{D}_{\mathrm{2SAT}} & \coloneqq \big (\{0,1\}; R_{00}, R_{01}, R_{10}, R_{11} \big)
    \\ \mathbb{D}_{\mathrm{HORNSAT}} & \coloneqq \big(\{0,1\}; R_{110}, R_{111}, \{0\}, \{1\}\big)
    \\\mathbb{D}_{3\mathrm{LIN}2} & \coloneqq \big (\{0,1\}; \text{all affine subspaces } R_{abcd} \text{ of } \mathbb{Z}_2^3\text{ of dimension } 2 \big)
\end{align*}    
where for all $a,b,c,d \in \{0,1\}$:
\begin{align*}
    R_{ab}&:=\{0,1\}^2\setminus\{(a,b)\}
    \\R_{abc}&:=\{0,1\}^3\setminus\{(a,b,c)\} 
    \\ R_{abcd}&:=\{(x,y,z)\in\mathbb{Z}_2^3 \ :\ ax + by + cz = d \}.
\end{align*}
 These structures are the relational counterparts of the atoms of $\Pboole$ in the sense that $[d_3]=\Pol(\mathbb{D}_{\mathrm{2SAT}})$, $[m]=\Pol(\mathbb{D}_{3\mathrm{LIN}2})$, and $[\wedge]=\Pol(\mathbb{D}_{\mathrm{HORNSAT}})$ (see, for instance,~\cite{Schaefer}).
\begin{Prop}\label{atoms} The following holds:
\begin{enumerate}
    \item $[\wedge]\mid[d_3]$.
    \item $[d_3]\mid[m]$.
    \item $[m]\mid[\wedge]$.
\end{enumerate}
\begin{proof}
(1) By definition, $d_3$ is a quasi majority operation. Let $f$ be any Boolean quasi majority operation. Then it is easy to check that $f$ does not preserve $R_{110}$ and thus $f\notin [\wedge]= \Pol(\mathbb{D}_{\mathrm{HORNSAT}})$. Hence, the quasi majority condition is a witness for  $[d_3]\npreceq[\wedge]$. 

We claim that the height 1 identity $f(x,y) \approx f(y,x)$ is a witness for $[\wedge]\npreceq[d_3]$. This identity is clearly satisfied by $\wedge$. 
Let $f$ be any Boolean binary commutative operation. Then $f$ preserves neither $R_{00}$ nor $R_{11}$. Hence, $f\notin[d_3] = \Pol(\mathbb{D}_{\mathrm{2SAT}})$ and thus the claim is proved.

(2) Let $f$ be a Boolean quasi majority operation. 
Then $f$ does not preserve $R_{1111} = \{(0,0,1), (0,1,0), (1,0,0), (1,1,1)\}$ and thus $f\notin [m] = \Pol(\mathbb{D}_{3\mathrm{LIN}2})$.
Hence, the quasi majority  condition is a witness for $[d_3]\npreceq [m]$. Let $g$ be any Boolean quasi minority operation, then $g$ does not preserve $R_{110}$ and therefore the quasi minority condition is a witness for $[m]\npreceq [d_3]$.

(3) Similar to case (1).
\end{proof}
\end{Prop}

\begin{Cor}\label{d3}
$[d_3,p] \npreceq [d_3]$.
\begin{proof}
If $[\wedge] \subseteq [d_3,p]$, then $[d_3,p]$ satisfies the height 1 identity $f(x,y) \approx f(y,x)$. In the proof of Proposition~\ref{atoms} it was shown that $[d_3]$ does not satisfy $f(x,y) \approx f(y,x)$. Hence, $[d_3,p]\npreceq[d_3]$.
\end{proof}
\end{Cor}

It is easy to check that $[d_3 , m] = \Pol({\mathbb{C}_2})$, where ${\mathbb{C}_2}$ is the relational structure ${\mathbb{C}_2} \coloneqq (\{0,1\};\{0\},\{1\},\{(0,1), (1,0)\})$.
 
\begin{Prop}\label{Pixley} The following holds:
\begin{enumerate}
    \item $[d_3 , m] \npreceq [d_3]$.
    \item $[d_3 , m] \npreceq [m]$.
    \item $[d_3 , m]\mid [\wedge]$.
\end{enumerate}
\begin{proof}
(1) and (2) follow immediately from Proposition~\ref{atoms}: the quasi minority condition is a witness for $[d_3,m]\npreceq[d_3]$ and the quasi majority condition is a witness for $[d_3,m]\npreceq [m]$. Concerning (3), it follows from Proposition~\ref{atoms} that the quasi majority condition is a witness for $[d_3 , m]\npreceq [\wedge]$.
Conversely, suppose that $g$ is a Boolean binary commutative operation. Then $g$ does not preserve ${\mathbb{C}_2}$. Therefore, the height 1 identity $f(x,y)\approx f(y,x)$ is a witness for $[\wedge]\npreceq [d_3,m]$. 
\end{proof}
\end{Prop}

We now prove that $\Pboole$ contains an infinite descending chain.
\begin{equation}\label{chain1}\tag{C1}
\overline{[d_3 , q]} \succ \overline{[d_4 , q]} \succ \overline{[d_5 , q]} \succ \dots \succ \overline{[q]}.
\end{equation}

In order to prove this fact, we introduce the following relational structures, also known as \emph{blockers}~\cite{oprsal2018taylors}: 
\begin{align*}
   \mathbb{B}_k \coloneqq ( \{0,1\} ; \{0\}, \{1\}, B_k) & \text{ where }
   B_k \coloneqq \{0,1\}^k \setminus (\underbrace{0,\dots,0}_{k})\\
   \mathbb B_\infty & \coloneqq \bigcup_{n\in\mathbb N} \mathbb B_n.
\end{align*}

Blockers are the relational counterparts of the clones considered in the chain \eqref{chain1}, because the same chain can be rewritten as:
\begin{equation*}
\overline{\Pol(\mathbb{B}_2)} \succ \overline{\Pol(\mathbb{B}_3)} \succ \overline{\Pol(\mathbb{B}_4)} \succ \dots \succ \overline{\Pol(\mathbb{B}_\infty)}.
\end{equation*}
We use the $\operatorname{QNU}$ identities to prove that the order is strict: in fact, $\Pol(\mathbb{B}_{n-1})$ satisfies $\operatorname{QNU}(n)$ but $\Pol(\mathbb{B}_n)$ does not. 
\begin{Prop}\label{QNU} For any natural number $n>2$, the quasi near-unanimity condition $\operatorname{QNU}(n)$ is a witness for 
$\Pol(\mathbb{B}_{n-1})\npreceq\Pol(\mathbb{B}_n)$.
\begin{proof}
Let $f$ be an $n$-ary quasi near-unanimity operation. Suppose for contradiction that $f$ is in $\Pol(\mathbb{B}_n)$. Note that $f$ is idempotent since it has to preserve the unary relations $\{0\}$ and $\{1\}$. In the following $(n\times n)$-matrix every column is an element of $B_n$. Then we get a contradiction since, applying $f$ row-wise, we obtain the missing $n$-tuple $(0,\dots,0)$.
\begin{equation*}
f
\begin{bmatrix} 
0 &\dots & 0 & 1 
\\ \vdots & \reflectbox{$\ddots$}& 1 & 0 
\\ 0 & \reflectbox{$\ddots$} & \reflectbox{$\ddots$} & \vdots 
\\ 1 & 0 & \dots & 0
\end{bmatrix} =
\begin{bmatrix}
0 \\\vdots\\\vdots\\0
\end{bmatrix}
\end{equation*}
Let $g$ be an $n$-ary operation defined as:
\begin{equation*}
    g(x_1,\dots,x_n)\coloneqq 
    \begin{cases}
    1 & \mbox{ if at least two variables are evaluated to 1}.
    \\ 0 & \mbox{ otherwise}.
    \end{cases}
\end{equation*}
Then, by definition, $g$ is a $\operatorname{QNU}$ operation. We claim that $g \in\Pol(\mathbb{B}_{n-1})$. Indeed, from an analysis of any $(n-1)\times n$-matrix $\mathbf M$ such that applying $g$ to the rows of $\mathbf M$ one gets the tuple $(0,\dots,0)$, we conclude that one of the columns of $\mathbf M$ must be equal to a $0$-vector. Thus the claim follows.
\end{proof}
\end{Prop}
With the same argument we can prove that there is another infinite descending chain, namely 
\begin{equation}\label{chain2}\tag{C2}
\overline{[d_3 , p]} \succ \overline{[d_4]} \succ \overline{[d_5]} \succ \dots \succ \overline{[p]}.
\end{equation}
Again we consider the relational counterparts of the clones involved in $\eqref{chain2}$ and rewrite the chain as follows:
\begin{equation*}
\overline{\Pol(\mathbb{B}_2, \leq)} \succ \overline{\Pol(\mathbb{B}_3, \leq)} \succ \overline{\Pol(\mathbb{B}_4, \leq)} \succ \dots \succ \overline{\Pol(\mathbb{B}_\infty, \leq)}.
\end{equation*}

\begin{Prop}\label{QNU2}
For any natural number $n>2$, the quasi near-unanimity condition $\operatorname{QNU}(n)$ is a witness for 
$\Pol(\mathbb{B}_{n-1},\leq)\npreceq\Pol(\mathbb{B}_n,\leq)$.
\end{Prop}
The next step is to show that $\eqref{chain1}$ and $\eqref{chain2}$ are two distinct chains in $\Pboole$. In particular, we prove that there is no minor-preserving map from $[q]$ to $[d_3,p]$. The height 1 condition which we use as a witness of this fact is the quasi-version of a celebrated set of identities from universal algebra~\cite{HagemannMitschke}.

\begin{Def}
The following set of height 1 identities
\begin{align*}
	p_0(x,y,z)& \approx p_0(x,x,x)
	\\ p_n(x,y,z)& \approx p_n(z,z,z), \text{ and }
	\\ p_i(x,x,y)& \approx p_{i+1}(x,y,y) \mbox{\ \ for every\ \ } i \leq n
\end{align*}
is called \emph{$n$-ary quasi Hagemann-Mitschke condition}, $\operatorname{HM}(n)$.
\end{Def}

\begin{Prop}\label{HM}
The height 1 condition $\operatorname{HM}(3)$ is a witness for $[q]\npreceq[d_3,p]$.

\begin{proof}
Note that $[q]\models\operatorname{HM}(3)$; defining:
\begin{align*}
	p^{[q]}_0(x,y,z)& \coloneqq q(x,x,x)
	\\ p^{[q]}_1(x,y,z)& \coloneqq q(x,y,z)
	\\ p^{[q]}_2(x,y,z)& \coloneqq q(z,x,y)
	\\ p^{[q]}_3(x,y,z)& \coloneqq q(z,z,z). 
\end{align*}
 
Suppose for contradiction that $\mathcal C\coloneqq\Pol(\mathbb{B}_2, \leq)$
satisfies the height 1 condition $\operatorname{HM}(3)$ via $p^{\mathcal C}_0,\dots,p^{\mathcal C}_3$. Then we have
\begin{equation*}
1 = p^{\mathcal C}_0(1,1,0) = p^{\mathcal C}_1(1,0,0) \leq p^{\mathcal C}_1(1,1,0) = \dots = p^{\mathcal C}_3(1,0,0) = 0
\end{equation*}
which is a contradiction.
\end{proof}
\end{Prop}

\begin{Prop}\label{min-order}
 $[m] \npreceq [d_3,p]$.
\end{Prop} 
\begin{proof}
We claim that the quasi minority condition is a witness for $[m]\npreceq [d_3,p]$. In fact, let $f$ be any Boolean quasi minority operation. Then $f$ does not preserve $B_2$ since the missing tuple $(0,0)$ can be obtained by applying $f$ to tuples in $B_2$. Hence, $f \notin [d_3,p] = \Pol(\mathbb{B}_2, \leq)$ and thus the claim follows.
\end{proof} 

\begin{Cor}\label{incchain}
Let $\mathcal C$ be a Boolean clone such that $[p] \subseteq \mathcal C \subseteq [d_3,p]$. Then $\mathcal C\mid [m]$.
\begin{proof}
Let $\mathcal C$ be as in the hypothesis. Let us suppose that $[m] \preceq \mathcal C$. Then we get $[m] \preceq \mathcal C \preceq [d_3,p]$, contradicting Proposition~\ref{min-order}. Let us suppose now that $\mathcal C \preceq [m]$. Then we get $[\wedge] \prec [p]\preceq \mathcal C\preceq [m]$, contradicting Proposition~\ref{atoms}.
\end{proof}
\end{Cor}

\begin{Cor}
$[m] \npreceq [d_3,q]$.
\begin{proof}
Since $[d_3,q]=\Pol(\mathbb{B}_2)$, the  argument is essentially the same as the one of Proposition~\ref{min-order}.
\end{proof}
\end{Cor}
\begin{Cor}\label{incompt3}
Let $\mathcal C$ be a Boolean clone such that $[q] \subseteq \mathcal C \subseteq [d_3,q]$. Then $\mathcal C\mid [m]$.
\begin{proof}
The proof is essentially the same as the one of Corollary~\ref{incchain}.
\end{proof}
\end{Cor}

\begin{Prop}\label{incompd3}
Let $\mathcal C$ be a Boolean clone such that $[\wedge] \subseteq \mathcal C \subseteq [d_4,q]$. Then $\mathcal C\mid [d_3]$.
\begin{proof}
Since $[\wedge] \subseteq \mathcal C$ and since, by Proposition~\ref{atoms}, we have $[\wedge]\npreceq [d_3]$, it follows that $\mathcal C\npreceq [d_3]$. Since $d_3$ is the Boolean majority operation, we have that $[d_3]$ satisfies the quasi majority condition $\operatorname{QNU}(3)$. By Proposition~\ref{QNU}, we have that $[d_4,q]$ does not satisfy $\operatorname{QNU}(3)$ and hence no subclone of $[d_4,q]$ satisfies $\operatorname{QNU}(3)$. Since $\mathcal C$ is a subclone of $[d_4,q]$ it follows that the quasi majority condition is a witness for $\mathcal C\npreceq [d_3]$.
\end{proof}
\end{Prop}

\begin{Prop}\label{coatom}
 $[0] \npreceq [m,q]$.
 \begin{proof}
Note that $[m,q] = \Pol(\{0,1\};\{0\},\{1\})$ is the clone of all the idempotent operations on $\{0,1\}$. Hence, $[m,q]$ contains no constants and thus $[m,q]\not\models f(x) \approx f(y)$ while $[0]$ does.
\end{proof}
\end{Prop}

\subsection{The Final Picture}\label{picture}
Putting all the results of the previous section together, we display an order diagram of $\Pboole$. We then use this diagram to revisit 
the complexity of Boolean CSPs. 
In Figure~\ref{fig:finalpicture} we indicate for each element of $\Pboole$ the corresponding complexity class.

\begin{Thm}\label{mainThm}
The pp-constructability poset restricted to the case of Boolean clones is the lattice $\Pboole$ in Figure~\ref{fig:finalpicture}.
\begin{proof}
Recall that every element in $\Pboole$ is a $\equiv$-class; for every $\equiv$-class we list explicitly the clones on $\{0,1\}$ that are in the considered class. The list is justified by the results proved in Section~\ref{collapses}:
    \begin{align*}
        \overline{[\emptyset]} & = \{[\emptyset], [c]\}
        & \overline{[\wedge]} & = \{[\wedge], [\vee]\}
        \\ \overline{[d_3]} & = \{[d_3]\}
        & \overline{[m]} & = \{[m], [m,c]\}
        \\ \overline{[d_3, m]} & = \{[d_3, m], [d_3,c]\}
        & \overline{[p]} & = \{[p], [p^\Delta]\}
        \\ \overline{[q]} & = \{[q], [q^\Delta]\}
        & \overline{[d_3, p]} & = \{[d_3, p], [d_3, p^\Delta], [\wedge,\vee]\}
        \\ \overline{[d_i, p]} & = \{[d_i,p], [d_i,p^\Delta] \mid i > 3\}
        & \overline{[d_i, q]} & = \{[d_i,q], [d_i,q^\Delta] \mid i \geq 3 \}
        \\ \overline{[m, q]} & = \{[\vee, q]\}
        & \overline{[0]} & = \{\mathcal C\mid 0 \in  \mathcal C\text{ or }  1 \in  \mathcal C\}
    \end{align*}

Note that all the clones of Post's lattice appear in this list. We have to show that there are no further collapses. Recall that if $\mathcal C$ and $\mathcal D$ are elements of Post's lattice such that $\mathcal C \subseteq \mathcal D$, then $\mathcal C\preceq\mathcal D$. Using this remark together with the results proved in Section~\ref{ourlattice} we get the following inequalities.
\begin{align*}
    [\emptyset] \preceq \mathcal C \preceq [m,q] \prec [0], \text{ for every } \mathcal C \neq [0] && \text{(Propositions~\ref{bottom}, \ref{coatom}, \ref{top})}
    \\ [d_3] \prec [d_3,m]\ \text{   and   }\ [m] \prec [d_3,m] && (\text{Proposition } \ref{Pixley})
    \\ [\wedge] \prec [p] \prec [q] && \text{(Propositions \ref{prop:jon},\ \ref{HM})}
    \\ [d_3] \prec [d_3,p] && (\text{Corollary } \ref{d3})
    \\ [p] \prec [d_{i+1},p] \prec [d_i,p], \text{ for every } i \geq 3 && (\text{Proposition } \ref{QNU2})
    \\ [q] \prec [d_{i+1},q] \prec [d_i,q], \text{ for every } i \geq 3 && (\text{Proposition } \ref{QNU})
    \\ [d_i,p] \prec [d_i,q], \text{ for every } i \geq 3 && (\text{Proposition } \ref{HM})
\end{align*}

It remains to prove that there are no other comparable elements in $\Pboole$. Propositions~\ref{atoms}, \ref{Pixley}, \ref{incompd3},  Corollary~\ref{incchain}, and Corollary~\ref{incompt3} ensure that this is the case.
\end{proof}
\end{Thm}
\begin{figure}
	\vspace{-3cm}
\centering
		\includegraphics[scale=0.4]{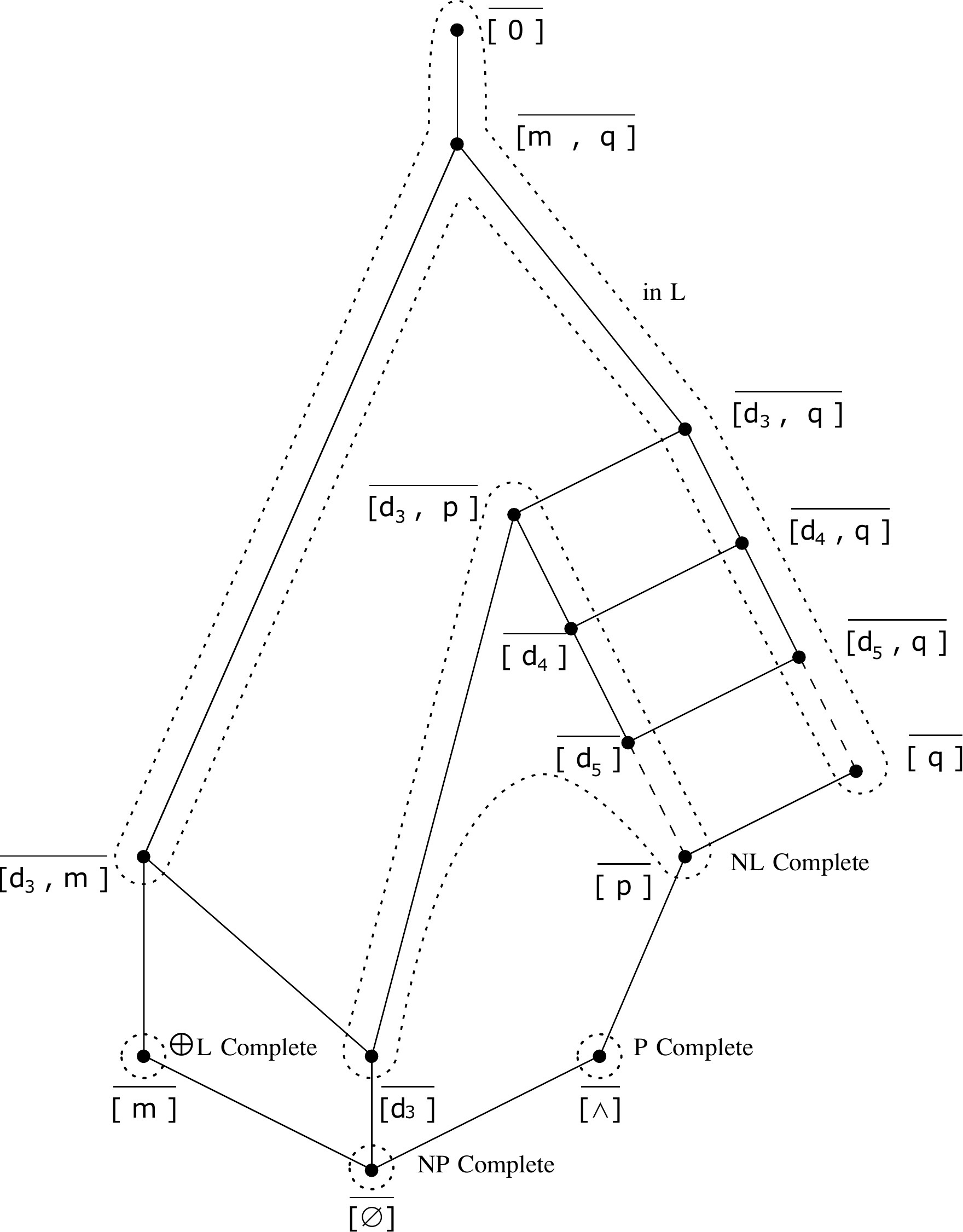}
		\caption{The lattice $\Pboole$}
		\label{fig:finalpicture}
\end{figure}

We already pointed out that the bottom element of $\Pboole$ represents the class of all the Boolean relational structures $\mathbb B$ such that $\operatorname{CSP}(\mathbb B)$ is NP-complete,
and Schaefer's theorem~\cite{Schaefer} implies that the CSP of every other Boolean structure is in P. Following~\cite{AllenderSchaefer}, we describe the complexity of Boolean CSPs within P. Combining Theorem~\ref{mainThm} with the main result in~\cite{AllenderSchaefer} we obtain the following. 

\begin{Thm}
Let $\bA$ be a relational structure with domain $\{0,1\}$ and finite relational signature.
\begin{itemize}
	\item If $\Pol(\bA) \equiv [\emptyset]$, then $\operatorname{CSP}(\bA)$ is NP-complete.
	\item If $\Pol(\bA) \equiv [\wedge]$, then $\operatorname{CSP}(\bA)$ is P-complete.
	\item If $\Pol(\bA) \equiv  [m]$, then $\operatorname{CSP}(\bA)$ is $\oplus$L-complete.
	\item If $\Pol(\bA) \equiv [d_3]$ or $[p] \preceq \Pol(\bA) \preceq [d_3,p]$, then $\operatorname{CSP}(\bA)$ is NL-complete.
	\item If $[d_3,m] \preceq \Pol(\bA)$ or $[q] \preceq \Pol(\bA)$, then $\operatorname{CSP}(\bA)$ is in L.
\end{itemize}
\end{Thm}
The same complexity results can be reached using general results collected in the survey article~\cite{Polymorphisms}.

\section{Concluding Remarks and Open Problems}\label{Open problems}
In this article we completely described $\Pboole$, i.e., the pp-constructability poset restricted to structures over a two-element set; equivalently, we studied the poset that arises from ordering Boolean clones with respect to the existence of minor-preserving maps. 

The natural next step is to study the pp-constructability poset on larger classes of finite structures.
Janov and Mu\v{c}nik~\cite{3elem} showed that there are continuum many clones over a three-element set, but all the clones considered 
in their proof have a constant operation, so they only correspond to a single element in our poset.
Uncountably many idempotent clones over a three element set have been described by Zhuk~\cite{Zhuk15}. We have not yet been able to separate them with height 1 conditions. Hence, there is still hope that $\Pfin$ has only countably many elements.
While it is easy to construct infinite antichains in $\Pfin$, we also do not know whether $\Pfin$ contains infinite ascending chains. 

It is easy to see that $\Pfin$ is a meet-semilattice: if $\mathcal C$ and $\mathcal D$ are clones on a finite set, then $\mathcal C \times \mathcal D$ is a clone that projects both to
$\mathcal C$ and to $\mathcal D$ via minor-preserving maps, and all other clones with this property have a minor-preserving map to $\mathcal C \times \mathcal D$. However, we do not know
whether $\Pfin$ is a lattice. It is known that $\mathbb K_3$ pp-constructs (even \emph{pp-interprets}; see, e.g.,~\cite{Bodirsky-HDR}) all finite structures. Hence, $\overline{\mathbb K_3}$ is the bottom element in $\Pfin$. 
Moreover, it can be shown that $\Pfin$ has no atoms and that $\overline{[m,q]}$ is the only coatom in $\Pfin$. Note that for every finite set A with at least two elements, the class $\overline{[m,q]}$
contains the clone of all idempotent operations on $A$. One can see that when studying $\Pfin$ we can focus on idempotent clones. 
However, the study 
of the entire poset $\Pfin$ is ongoing and will be the topic of a future publication. 

\bibliographystyle{abbrv}

\begin{thebibliography}{10}

\bibitem{AllenderSchaefer}
E.~Allender, M.~Bauland, N.~Immerman, H.~Schnoor, and H.~Vollmer.
\newblock The complexity of satisfiability problems: Refining {S}chaefer's
  theorem.
\newblock {\em Journal of Computer and System Sciences}, 75(4):245--254, 2009.

\bibitem{Polymorphisms}
L.~Barto, A.~Krokhin, and R.~Willard.
\newblock {Polymorphisms, and How to Use Them}.
\newblock In A.~Krokhin and S.~Zivny, editors, {\em The Constraint Satisfaction
  Problem: Complexity and Approximability}, volume~7 of {\em Dagstuhl
  Follow-Ups}, pages 1--44. Schloss Dagstuhl--Leibniz-Zentrum fuer Informatik,
  Dagstuhl, Germany, 2017.

\bibitem{wonderland}
L.~Barto, J.~Opr\v{s}al, and M.~Pinsker.
\newblock The wonderland of reflections.
\newblock {\em Israel Journal of Mathematics}, 223(1):363--398, 2018.

\bibitem{Bodirsky-HDR}
M.~Bodirsky.
\newblock Complexity classification in infinite-domain constraint satisfaction.
\newblock M\'emoire d'habilitation \`a diriger des recherches, Universit\'{e}
  Diderot -- Paris 7. Available at arXiv:1201.0856, 2012.

\bibitem{BulatovHColoring}
A.~A. Bulatov.
\newblock {H}-coloring dichotomy revisited.
\newblock {\em Theoretical Computer Science}, 349(1):31--39, 2005.

\bibitem{GarciaTaylor}
O.~Garc{\'\i}a, W.~Taylor, and W.~Taylor.
\newblock {\em The Lattice of Interpretability Types of Varieties}.
\newblock American Mathematical Society: Memoirs of the. American Mathematical
  Society, 1984.

\bibitem{HagemannMitschke}
J.~Hagemann and A.~Mitschke.
\newblock On $n$-permutable congruences.
\newblock {\em Algebra Universalis}, pages 8--12, 1973.

\bibitem{HellNesetril}
P.~Hell and J.~Ne\v{s}et\v{r}il.
\newblock On the complexity of {H}-coloring.
\newblock {\em Journal of Combinatorial Theory, Series B}, 48:92--110, 1990.

\bibitem{3elem}
Y.~I. Janov and A.~A. Mu\v{c}nik.
\newblock On the existence of $k$-valued closed classes that have no bases.
\newblock {\em Dokl. Akad. Nauk SSSR}, 127:44--46, 1959.

\bibitem{Neumann}
W.~D. Neumann.
\newblock On {M}al'cev conditions.
\newblock {\em J. Aus. Math. Soc.}, 17(3), 1974.

\bibitem{oprsal2018taylors}
J.~Opr\v{s}al.
\newblock Taylor's modularity conjecture and related problems for idempotent
  varieties.
\newblock {\em Order}, 35(3):433--460, 11 2018.

\bibitem{Post}
E.~L. Post.
\newblock The two-valued iterative systems of mathematical logic.
\newblock {\em Annals of Mathematics Studies}, 5, 1941.

\bibitem{Schaefer}
T.~J. Schaefer.
\newblock The complexity of satisfiability problems.
\newblock In {\em Proceedings of the Symposium on Theory of Computing (STOC)},
  pages 216--226, 1978.

\bibitem{Zhuk15}
D.~Zhuk.
\newblock The lattice of all clones of self-dual functions in three-valued
  logic.
\newblock {\em Multiple-Valued Logic and Soft Computing}, 24(1-4):251--316,
  2015.

\end{thebibliography}
\def\cprime{$'$} \def\cprime{$'$}

\end{document}